\documentclass[letterpaper, 10pt, conference]{ieeeconf}      

\IEEEoverridecommandlockouts                              
\usepackage{amsmath,graphicx,amsfonts,amssymb,epsfig,mathrsfs}
\usepackage{subcaption}
\usepackage{color}
\usepackage{multirow}
\usepackage{rotating}
\usepackage{graphicx}%
\usepackage{algorithm}
\usepackage{algpseudocode}
\usepackage{algorithmicx}
\usepackage{cite}
\usepackage{url}
\usepackage{framed}
\usepackage{bm}

\newtheorem{theorem}{Theorem}
\newtheorem{corollary}[theorem]{Corollary}
\newtheorem{definition}{Definition}

\newtheorem{proposition}{Proposition}

\usepackage{tikz}
\usetikzlibrary{calc,trees,positioning,arrows,chains,shapes.geometric,decorations.pathreplacing,decorations.pathmorphing,shapes, matrix,shapes.symbols}

\DeclareMathOperator*{\argmax}{arg\:max}
\DeclareMathOperator*{\argmin}{arg\:min}
\DeclareMathOperator*{\arginf}{arg\:inf}

\newcommand{\prox}{\rm{prox}}

\newcommand{\bw}{{\bm{w}}}
\newcommand{\bx}{{\bm{x}}}
\newcommand{\brh}{{\bm{\varrho}}}

\renewcommand{\dagger}{{T}}
\newcommand{\differential}{{\rm{d}}}

\allowdisplaybreaks

\title{\LARGE\textbf{
Proximal Recursion for Solving the Fokker-Planck Equation}
}

\author{Kenneth F. Caluya, Abhishek Halder
\thanks{Kenneth F. Caluya, and Abhishek Halder are with the Department of Applied Mathematics, University of California, Santa Cruz, CA 95064, USA,
        {\tt\small{\{kcaluya,ahalder\}@ucsc.edu}}%
}}

\begin{document}

\maketitle
\thispagestyle{empty}
\pagestyle{empty}

\begin{abstract}
We develop a new method to solve the Fokker-Planck or Kolmogorov's forward equation that governs the time evolution of the joint probability density function of a continuous-time stochastic nonlinear system. Numerical solution of this equation is fundamental for propagating the effect of initial condition, parametric and forcing uncertainties through a nonlinear dynamical system, and has applications encompassing but not limited to forecasting, risk assessment, nonlinear filtering and stochastic control. Our methodology breaks away from the traditional approach of spatial discretization for solving this second-order partial differential equation (PDE), which in general, suffers from the ``curse-of-dimensionality". Instead, we numerically solve an infinite dimensional proximal recursion in the space of probability density functions, which is theoretically equivalent to solving the Fokker-Planck-Kolmogorov PDE. We show that the dual formulation along with the introduction of an entropic regularization, leads to a smooth convex optimization problem that can be implemented via suitable block co-ordinate iteration and has fast convergence due to certain contraction property that we establish. This approach enables meshless implementation leading to remarkably fast computation.
\end{abstract}


\section{Introduction}
Given a deterministic or stochastic dynamical system in continuous time over some finite dimensional state space, say $\mathbb{R}^{n}$, we consider the problem of propagating the trajectory ensembles or densities subject to stochastic initial conditions -- often referred to as the \emph{belief or uncertainty propagation} problem. Mathematically, this amounts to solving an initial value problem associated with a partial differential equation (PDE) of the form
\begin{eqnarray}
\displaystyle\frac{\partial\rho}{\partial t} = \mathcal{L}\rho, \quad \rho(\bm{x},t=0) = \rho_{0}(\bm{x}) \;\text{given}, \label{generalPDE}	
\end{eqnarray}
describing the transport of the density function $\rho(\bm{x},t)$, which is a function of the state vector $\bm{x}\in\mathbb{R}^{n}$, and time $t\geq 0$. Here, $\mathcal{L}$ is a spatial operator that guarantees $\rho \geq 0$, and $\int_{\mathbb{R}^{n}}\rho(\bm{x},t){\mathrm{d}}\bm{x}=1$ for all $t\geq 0$. Without loss of generality, one can interpret $\rho(\bm{x},t)$ as the joint probability density function (PDF) of the state vector $\bm{x}$ at time $t$. We refer to (\ref{generalPDE}) as \emph{transport PDE}.

The structural form of $\mathcal{L}$ in (\ref{generalPDE}) depends on the underlying trajectory level dynamics. For example, consider the case when the dynamics of $\bm{x}(t)\in\mathbb{R}^{n}$ is governed by an ordinary differential equation (ODE) $\dot{\bm{x}} = \bm{f}(\bm{x},t)$, subject to random initial condition $\bm{x}(t=0) = \bm{x}_{0}$ with known joint PDF $\rho_{0}$ (for notational ease, we write $\bm{x}_{0} \sim \rho_{0}$). Then, $\mathcal{L}\rho \equiv -\nabla\cdot\left(\rho\bm{f}\right)$, where $\nabla$ denotes the gradient with respect to (w.r.t.) the standard Euclidean metric, and the resulting \emph{first order} transport PDE is known as the \emph{Liouville equation}. 

More generally, consider the case when the dynamics of $\bm{x}(t)\in\mathbb{R}^{n}$ is governed by an It\^{o} stochastic differential equation (SDE) $\differential\bx = f\left(\bm{x},t\right)\differential t \: + \: \bm{g}(\bm{x},t)\:\differential\bw$, $\bm{x}(t=0)=\bm{x}_{0} \sim \rho_{0}$ (given), the process noise $\bw(t)\in\mathbb{R}^{m}$ is Wiener and satisfy $\mathbb{E}\left[\differential w_{i} \differential w_{j}\right] = \delta_{ij}\differential t$ for all $i,j=1,\hdots,n$, where $\delta_{ij}=1$ for $i=j$, and zero otherwise. Then, 
\begin{eqnarray}
	\mathcal{L}\rho \equiv -\nabla\cdot\left(\rho\bm{f}\right) + \frac{1}{2}\sum_{i,j=1}^{n}\frac{\partial^{2}}{\partial x_{i}\partial x_{j}}(\rho\bm{g}\bm{g}^{\top})_{ij}, 
	\label{FPKoperator}
\end{eqnarray} 
and the resulting  transport PDE (\ref{FPKoperator}) is known as the \emph{Fokker-Planck} or \emph{Kolmogorov's forward equation}. Hereafter, we will refer it as the FPK PDE.

The problem of uncertainty propagation, that is, the problem of computing $\rho(\bm{x},t)$  that satisfies a PDE of the form (\ref{FPKoperator}), is ubiquitous across science and engineering. Representative applications include meteorological forecasting \cite{ehrendorfer1994liouville}, dispersion analysis in spacecraft entry-descent-landing \cite{halder2011dispersion}, orientation density evolution for liquid crystals in chemical physics \cite{hess1976fokker,muschik1997mesoscopic,kalmykov1998analytical}, motion planning in robotics \cite{park2005diffusion,park2008kinematic,hamann2008framework}, computing the prior PDF in nonlinear filtering\cite{challa2000nonlinear,daum2005nonlinear}, probabilistic model validation \cite{halder2011model,halder2012further, halder2014probabilistic}, and analyzing the statistical mechanics of macromolecules \cite{flory1969statistical}. In all these applications, it is of importance to compute the joint PDF $\rho(\bm{x},t)$ in a scalable and unified manner, rather than employing specialized techniques in a case-by-case basis or developing discretization-based PDE solvers which suffer from the ``curse of dimensionality" \cite{bellman1957dynamic}. The Liouville PDE being first order, can be solved efficiently using the method-of-characteristics \cite{halder2011dispersion}. However, solving the second order FPK PDE in a manner that avoids both spatial discretization and function approximation, remains challenging to date.

In this paper, we pursue the solution of (\ref{generalPDE}) through a variational viewpoint arising from the theory of optimal mass transport \cite{villani2003topics}. This viewpoint, first proposed in \cite{jordan1998variational}, interprets (\ref{generalPDE}) as a gradient or steepest descent of certain functional $\Phi(\cdot)$ on the infinite dimensional manifold of PDFs with finite second (raw) moments, denoted as\footnote{We denote the expectation operator w.r.t. the measure $\rho(\bx)\differential\bx$ as $\mathbb{E}_{\rho}\left[\cdot\right]$.} 
\[\mathscr{D}_{2} := \{\rho:\mathbb{R}^{n}\mapsto\mathbb{R} \mid \rho \geq 0, \int_{\mathbb{R}^{n}}\rho=1, \:\mathbb{E}_{\rho}[\bm{x}^{\top}\bm{x}] < \infty\}.\] 
Specifically, let $k=0,1,2,\hdots$, and for some fixed time-step $h>0$, consider a variational recursion
\begin{eqnarray}
\varrho_{k}(\bm{x}) = \underset{\varrho\in\mathscr{D}_{2}}{{\rm{arg\:inf}}}\:\frac{1}{2}\:d^{2}\left(\varrho,\varrho_{k-1}\right) \, + \,h\: \Phi(\varrho),\label{PDFGradDescent}	
\end{eqnarray}
subject to the initial condition $\varrho_{0}(\bm{x}) := \rho_{0}(\bm{x})$, i.e., the initial PDF of (\ref{generalPDE}). Here, $d(\cdot,\cdot)$ is a distance metric on the manifold $\mathscr{D}_{2}$. Then, the idea is to design the metric $d(\cdot,\cdot)$ and the functional $\Phi(\cdot)$ in (\ref{PDFGradDescent}) such that $\varrho_{k}(\bm{x}) \rightarrow \rho(\bm{x},t=kh)$ as $h\downarrow 0$, i.e., in the small time-step limit, the solution of the variational recursion (\ref{PDFGradDescent}) converges (in strong $L^{1}$ sense) to that of (\ref{generalPDE}). The main result in \cite{jordan1998variational} was to show that for FPK operators of the form (\ref{FPKoperator}) with $\bm{f}$ being a gradient vector field and $\bm{g}$ being a scalar multiple of identity matrix, the distance $d(\cdot,\cdot)$ can be taken as the Wasserstein-2 metric with $\Phi(\cdot)$ as the free energy functional. We will make these ideas precise in Section II and III. The resulting variational recursion (\ref{PDFGradDescent}) has since been known as the \emph{Jordan-Kinderlehrer-Otto (JKO) scheme} \cite{ambrosio2008gradient}, and we will refer the FPK operator with such assumptions on $\bm{f}$ and $\bm{g}$ to be in ``JKO canonical form". Similar gradient descent schemes have been derived for many other PDEs; see e.g., \cite{santambrogio2017euclidean} for a recent survey. 

To motivate gradient descent in infinite dimensional spaces, we appeal to a more familiar setting, i.e., gradient descent in $\mathbb{R}^n$ associated with the flow 
\begin{eqnarray}
\dfrac{\differential \bx}{\differential t } = -\nabla \varphi \left(\bx\right)\: \quad \bx(0) = \bx_{0}, 
\label{EGRflow}
\end{eqnarray}
where $\bx ,\bx_0\in \mathbb{R}^{n}$ and $\varphi : \mathbb{R}^{n} \rightarrow \mathbb{R}_{\geq 0}$, and is continuously differentiable. The Euler discretization for (\ref{EGRflow}) is given by
\begin{eqnarray}
\bx_{k} - \bx_{k-1} =-h \nabla \varphi(\bx_{k-1}), 
\end{eqnarray}
which can be rewritten as a variational recursion
\begin{eqnarray}
\bx_{k} = \underset{\bx} {{\rm{arg\:min}}}\: \frac{1}{2} \parallel \bx - \bx_{k-1} \parallel ^2 + h\: \varphi(\bx) + o(h).
\label{ExplicitEuler}
\end{eqnarray}
In the optimization literature, the mapping $\bx_{k-1} \mapsto \bx_{k}$, given by
\begin{eqnarray}
{\prox}^{\parallel \cdot \parallel}_{h\varphi}(\bx_{k-1}) := \underset{\bx} {{\rm{arg\:min}}}\: \frac{1}{2} \parallel \bx - \bx_{k-1} \parallel ^2 + h\: \varphi(\bx),
\label{FiniteDimProxOpDef}
\end{eqnarray}
is called the ``proximal operator" \cite[p. 142]{parikh2014proximal}. The sequence $\{\bx_{k}\}$ generated by the proximal recursion 
\begin{eqnarray}
\bx_{k} =   {\prox}^{\parallel \cdot \parallel}_{h\varphi}(\bx_{k-1}), \quad k=0,1,2,\hdots
\label{FiniteDimProxRecursion}
\end{eqnarray}
converges to the flow of the ODE (\ref{EGRflow}), i.e., the sequence satisfies $\bx_{k} \rightarrow \bx(t=kh)$ as the step-size $h\downarrow 0$. Using the finite dimensional viewpoint (\ref{FiniteDimProxOpDef}), we define
\begin{eqnarray}
{\prox}^{d^{2}}_{h\Phi}(\varrho_{k-1}) := \underset{\varrho\in\mathscr{D}_{2}}{{\rm{arg\:inf}}}\:\frac{1}{2}\:d^{2}\left(\varrho,\varrho_{k-1}\right) \, + \,h\: \Phi(\varrho),
\label{InfiniteDimProxOpDef}
\end{eqnarray}
as an infinite dimensional proximal operator. As mentioned above, the sequence $\{\varrho_{k}\}$ generated by the proximal recursion (\ref{PDFGradDescent}) converges to the flow of the PDE (\ref{EGRflow}), i.e., the sequence satisfies $\varrho_{k}(\bm{x}) \rightarrow \rho(\bm{x},t=kh)$ as the step-size $h\downarrow 0$. We also note that in the finite dimensional case,
\begin{eqnarray}
\frac{\differential}{\differential t} \varphi = \langle \nabla \varphi,  -\nabla \varphi \rangle = -\parallel \nabla \varphi \parallel ^2  <  0
\label{dpsidt}
\end{eqnarray}
which implies $\varphi$ decays along the flow of (\ref{EGRflow}). As we will see next, the appeal of using (\ref{PDFGradDescent}) to solve the FPK PDE comes from the fact that the Euclidean gradient descent can be generalized to the manifold $\mathscr{D}_2$ by appropriately choosing the metric $d(\cdot,\cdot)$ and the functional $\Phi(\cdot)$ in (\ref{PDFGradDescent}), in parallel with the quantities $\parallel \cdot \parallel$ and $\varphi(\cdot)$  in (\ref{FiniteDimProxRecursion}), respectively.

\begin{figure}[t]
\centering
\vspace*{0.05in}
\includegraphics[width=.89\linewidth]{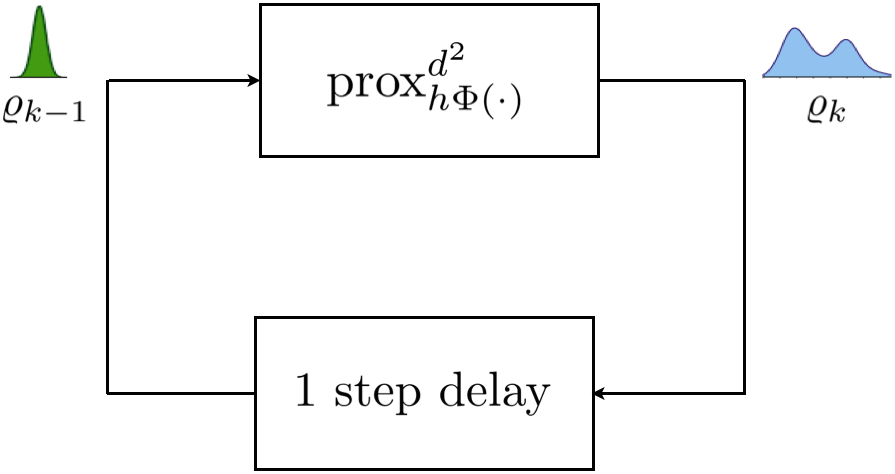}
\caption{\small{The JKO scheme can be described by successive evaluation of proximal operators to recursively update PDFs from time $t=(k-1)h$ to $t= kh$ for $k=1,2,\hdots$, and time-step $h>0$.}}
\vspace*{-0.25in}
\label{fig:ProxSchematic}
\end{figure}

In this paper, we will develop an algorithm to solve the FPK PDE via proximal recursion of the form (\ref{PDFGradDescent}) without making any spatial discretization. A schematic is shown in Fig. \ref{fig:ProxSchematic}. The resulting recursion is proved to be contractive and enjoys fast numerical implementation. Numerical simulation results show the efficacy of the proposed formulation. 


\section{Preliminaries}
In the following, we provide the definitions of the Kullback-Leibler divergence, and the 2-Waserstein metric, which will be useful in the sequel. We also point out some notations used throughout this paper.
\begin{definition} The Kullback-Leibler divergence between two probability measures $\differential \pi_{i}(\bm{x}) =\rho_{i}(\bm{x})\differential\bm{x}$, $i=\{1,2\}$, is given by
\begin{eqnarray}
{\rm{D}}_{{\rm{KL}}}\left(\differential\pi_{1}\parallel\differential\pi_{2}\right) := \displaystyle\int\rho_{1}(\bm{x})\log\displaystyle\frac{\rho_{1}(\bm{x})}{\rho_{2}(\bm{x})}\:\differential\bm{x}, \label{KLdivdefn}	
\end{eqnarray}
which is non-negative, and vanishes if and only if $\rho_{1}=\rho_{2}$. However, (\ref{KLdivdefn}) is not a metric since it is neither symmetric, nor does it satisfy the triangle inequality.
\end{definition}

\begin{definition} The 2-Wasserstein metric between two probability measures $\differential \pi_{1}(\bm{x}) =\rho_{1}(\bm{x})\differential\bm{x}$ and $\differential\pi_{2}(\bm{y}) =\rho_{2}(\bm{y})\differential\bm{y}$ supported respectively on $\mathcal{X},\mathcal{Y}\subseteq \mathbb{R}^{n}$, is denoted as $W(\pi_{1},\pi_{2})$ (equivalently, $W\left(\rho_{1},\rho_{2}\right)$ whenever $\pi_{1},\pi_{2}$ are absolutely continuous so that the PDFs $\rho_{1},\rho_{2}$ exist), and arises in the theory of optimal mass transport \cite{villani2003topics}; it is defined as
\begin{align}
\notag &W(\pi_{1},\pi_{2}):= \\
                               & \left(\underset{\differential\pi\in\Pi\left(\pi_{1},\pi_{2}\right)}{\inf}\displaystyle\int_{\mathcal{X}\times\mathcal{Y}}\parallel \bm{x} - \bm{y} \parallel_{2}^{2}\:\differential\pi\left(\bm{x},\bm{y}\right) \right)^{\frac{1}{2}},
\label{Wdefn}	
\end{align}
where $\Pi\left(\pi_{1},\pi_{2}\right)$ denotes the collection of all probability measures on the product space $\mathcal{X}\times\mathcal{Y}$ having finite second moments, with marginals $\pi_{1}$ and $\pi_{2}$, respectively. Its square, $W^{2}(\pi_{1},\pi_{2})$ equals \cite{benamou2000computational} the minimum amount of work required to transport $\pi_{1}$ to $\pi_{2}$ (or equivalently, $\rho_{1}$ to $\rho_{2}$). It is well-known \cite[Ch. 7]{villani2003topics} that $W(\pi_{1},\pi_{2})$ defines a metric on the manifold $\mathscr{D}_{2}$.
\end{definition}

\subsubsection*{Notations} 
Throughout the paper, we will use bold-faced capital letters for matrices and bold-faced lower-case letters for column vectors. We use the symbol $\langle \cdot, \cdot\rangle$ to denote the Euclidean inner product. In particular, $\langle \bm{A},\bm{B} \rangle:= {\mathrm{trace}}({\bm{A}}^{\top}\bm{B})$ denotes
Frobenius inner product between matrices $\bm{A}$ and $\bm{B}$, and $\langle \bm{a},\bm{b}\rangle := \bm{a}^{\top}\bm{b}$ denotes the inner product between column vectors $\bm{a}$ and $\bm{b}$. We use $\mathcal{N}(\mu,\sigma^2)$ to denote a univariate Gaussian PDF with mean $\mu$ and variance $\sigma^2$. Likewise, $\mathcal{N}(\bm{\mu},\bm{\Sigma})$ denotes a multivariate Gaussian PDF with mean vector $\bm{\mu}$ and covariance matrix $\bm{\Sigma}$. The operands $\log(\cdot)$, $\exp(\cdot)$ and $\geq 0$ are to be understood as element-wise. The notations $\odot$ and $\oslash$ denote element-wise (Hadamard) product and division, respectively. We use $\bm{I}_{n}$ to denote the $n\times n$ identity matrix. The symbols $\bm{1}$ and $\bm{0}$ stand for column vectors of appropriate dimension containing all ones, and all zeroes, respectively.
 


\section{JKO Canonical Form}


In this paper, we consider the It\^{o} SDE
\begin{eqnarray}
\differential\bx = -\nabla \psi\left(\bx\right)\differential t \: + \: \sqrt{2\beta^{-1}}\:\differential\bw	, \quad \bx(0) = \bx_{0},
\label{ItoGradient}
\end{eqnarray}
where the time $t \in [0,\infty)$, the state vector $\bx\in\mathbb{R}^{n}$, the drift potential $\psi : \mathbb{R}^{n} \mapsto (0,\infty)$, the diffusion coefficient $\beta > 0$, and the initial condition $\bx_{0} \sim \rho_{0}(\bx)$. For the sample path $\bx(t)$ dynamics given by the SDE (\ref{ItoGradient}), the flow of the joint PDF $\rho\left(\bm{x},t\right)$ is governed by the FPK PDE
\begin{eqnarray}
\displaystyle\frac{\partial\rho}{\partial t} = \nabla\cdot\left(\rho\nabla \psi\right) \: + \: \beta^{-1}\Delta\rho, \quad \rho(\bx,0) = \rho_{0}(\bx), \label{FPKgradient}	
\end{eqnarray}
and its solution satisfies $\rho \geq 0$, $\int_{\mathbb{R}^{n}}\rho\:\differential\bx = 1$ for all $t\in[0,\infty)$. It is easy to verify that the unique stationary solution of (\ref{FPKgradient}) is the Gibbs PDF $\rho_{\infty}(\bm{x}) = \kappa \exp\left(-\beta \psi(\bm{x})\right)$, where the normalizing constant $\kappa:=\int_{\mathbb{R}^{n}}\exp(-\beta \psi(\bm{x}))$ is referred to as the \emph{partition function}.

A Lyapunov functional associated with the FPK PDE (\ref{FPKgradient}) is the \emph{free energy}
\begin{align}
F(\rho) :=& \: \mathbb{E}_{\rho}\left[\psi + \beta^{-1}\log\rho \right] \label{FPKFreeEnergy1}\\
=& \: \beta^{-1} D_{{\rm{KL}}}\left(\rho \parallel \exp\left(-\beta \psi(\bm{x})\right)\right)	\geq 0,
\label{FPKFreeEnergy2}
\end{align}
that decays \cite{jordan1998variational} along the solution trajectory of (\ref{FPKgradient}), i.e., $\frac{\differential}{\differential t}F < 0$. This follows from re-writing (\ref{FPKgradient}) as
\begin{eqnarray}
\displaystyle\frac{\partial{\rho}}{\partial t} = \nabla\cdot\left(\rho\nabla\zeta\right), \quad \text{where}\quad \zeta := \beta^{-1}\left(1 + \log\rho\right) + \psi,
\label{FPKadvectionform}	
\end{eqnarray}
and consequently
\begin{eqnarray}
\frac{\differential}{\differential t}F = -\mathbb{E}_{\rho}\left[\parallel \nabla\zeta \parallel^{2}\right] < 0,
\label{dFdt}	
\end{eqnarray}
with equality achieved at the stationary solution $\rho_{\infty} = \kappa e^{-\beta \psi(\bm{x})}$. In our context, (\ref{dFdt}) serves as the infinite-dimensional analog of (\ref{dpsidt}). The term \emph{free energy} is motivated by noting that (\ref{FPKFreeEnergy1}) can be seen as the sum of the \emph{potential energy} $\int_{\mathbb{R}^{n}}\psi(\bx)\rho\:\differential\bx$ and the \emph{internal energy} $\beta^{-1}\int_{\mathbb{R}^{n}}\rho\log\rho\:\differential\bx$. When $\psi=0$, the PDE (\ref{FPKgradient}) reduces to the heat equation, which by (\ref{FPKFreeEnergy1}), can then be interpreted as an entropy maximizing flow. 

The seminal paper \cite{jordan1998variational} establishes that the FPK PDE (\ref{FPKgradient}) can be seen as the gradient descent flow of the \emph{free energy} functional $F(\cdot)$ w.r.t. the 2-Wasserstein Metric. Specifically, the solution of (\ref{FPKgradient})
can be recovered from the following proximal recursion of the form (\ref{PDFGradDescent}):
\begin{subequations}
\begin{align}
\varrho_{k}      &= {\prox}^{W^2}_{hF(\cdot)} (\varrho_{k-1}) \label{JKOscheme:a} \\
 &=  \underset{\varrho \in \mathscr{D}_2 }\arginf \ \frac{1}{2} W^2(\varrho_{k-1},\varrho) + h\: F(\varrho), \; k= 1,2,\hdots \label{JKOscheme:b}
 \end{align}
 \label{JKOscheme}
\end{subequations} 
with $\varrho_{0} \equiv \rho_{0}(\bx)$ (from (\ref{FPKgradient})) as $h\downarrow 0$. Next, we develop a framework to numerically solve (\ref{JKOscheme}).


\section{Main Results}
To solve (\ref{JKOscheme}), we discretize time as $t=0,h,2h, \hdots$, and develop an algorithm to solve (\ref{JKOscheme}) without making any spatial discretization. In other words, we would like to perform the recursion (\ref{JKOscheme}) on weighted scattered point cloud $\{\bx_{k}^{i},\varrho_{k}^{i}\}_{i=1}^{N}$ of cardinality $N$ at $t_{k} = kh$, $k\in\mathbb{N}$, where the location of the point $\bx_{k}^{i}\in\mathbb{R}^{n}$ denotes the state-space coordinate, and the corresponding weight $\varrho_{k}^{i} \in \mathbb{R}_{\geq 0}$ denotes the value of the joint  PDF evaluated at that point at time $t_{k}$. Such weighted scattered point cloud representation of (\ref{JKOscheme}) results in the following problem: 
\begin{align}
 \brh_k  =\underset{\brh}\argmin \bigg\{  \underset{\bm{M }\in \Pi(\brh_{k-1},\brh)} \min \frac{1}{2}\langle \bm{C}_{k},\bm{M}\rangle + h \: \langle \bm{\psi}_{k-1}\notag\\
 +\beta^{-1}\log\brh,\brh \rangle   \bigg\},
 \label{FiniteSampleJKO}
\end{align}
to be solved for $k=1,2,\hdots$, where the drift potential vector $\bm{\psi}_{k-1} \in \mathbb{R}^{N}$ is given by
\[\bm{\psi}_{k-1}(i) := \psi\left(\bm{x}_{k-1}^{i}\right), \quad i=1,2,\hdots,N.\]
Similarly, the probability vectors $\brh, \brh_{k-1} \in \mathbb{R}^{N}$. Furthermore, for each $k=1,2,\hdots$, the matrix $\bm{C}_{k}\in\mathbb{R}^{N\times N}$ is given by
\[\bm{C}_{k}(i,j) := \parallel \bm{x}_{k}^{i} -  \bm{x}_{k-1}^{j} \parallel_{2}^{2}, \quad i,j=1,2,\hdots,N,\] 
and $\Pi(\brh_{k-1},\brh)$ stands for the set of all matrices $\bm{M}\in\mathbb{R}^{N\times N}$ such that 
\begin{align}
\bm{M}\geq 0, \quad \bm{M}\bm{1} = \brh_{k-1}, \quad \bm{M}^{\top}\bm{1} = \brh.
\label{CouplingMatrixConstr}
\end{align}
Due to the nested minimization structure in (\ref{FiniteSampleJKO}), its numerical solution is far from obvious. Notice that the inner minimization in  (\ref{FiniteSampleJKO}) is a standard linear programming problem if it were to be solved for a given $\brh$, as in the Monge-Kantorovich optimal mass transport \cite{villani2003topics}. However, the outer minimization in  (\ref{FiniteSampleJKO}) precludes a direct numerical approach.

To circumvent the aforesaid issues, following \cite{karlsson2017generalized}, we first regularize and then dualize (\ref{FiniteSampleJKO}). Specifically, adding an entropic regularization $H(\bm{M}) := \langle\bm{M}, \log\bm{M}\rangle$ in (\ref{FiniteSampleJKO}) yields
\begin{align}
\brh_k = \underset{\brh}\argmin \bigg\{  \underset{\bm{M} \in \Pi(\brh_{k-1},\brh)} \min \frac{1}{2}\langle \bm{C}_{k},\bm{M}\rangle +
\epsilon H(\bm{M}) \notag\\
+ h \: \langle \bm{\psi}_{k-1}+\beta^{-1}\log\brh,\brh \rangle   \bigg\},
\label{EntropyRegJKO}
\end{align}
where $\epsilon>0$ is a regularization parameter. The entropic regularization is standard in optimal mass transport literature \cite{cuturi2013sinkhorn,benamou2015iterative} and leads to efficient Sinkhorn iteration for the inner minimization. In our context, the entropic regularization ``algebrizes" the inner minimization in the sense if $\bm{\lambda}_{0},\bm{\lambda}_{1}$ are Lagrange multipliers associated with the equality constraints in (\ref{CouplingMatrixConstr}), then the optimal coupling matrix $\bm{M}^{\rm{opt}} := [m^{\rm{opt}}(i,j)]$ in (\ref{EntropyRegJKO}) has the Sinkhorn form
\begin{align}
m^{\rm{opt}}(i,j) = \exp\left(\bm{\lambda}_{0}(i)h/\epsilon\right) \exp\left(-\bm{C}_{k}(i,j)/(2\epsilon)\right) \notag\\
\exp\left(\bm{\lambda}_{1}(j)h/\epsilon\right).
\label{SinkhornFormInnerArgmin}	
\end{align}
Since the objective in (\ref{EntropyRegJKO}) is proper convex and lower semi-continuous in $\brh$, the \emph{strong duality} holds, and we consider the Lagrange dual of (\ref{EntropyRegJKO}) given by:
\begin{align}
\bm{\lambda}_0^{\rm{opt}},\bm{\lambda}_1^{\rm{opt}} 
=\underset{\bm{\lambda}_0,\bm{\lambda}_1\geq 0}\argmax \bigg \{  \langle \bm{\lambda}_0,\brh_{k-1}  \rangle 
- F^{\star}(-\bm{\lambda}_1) \notag\\
-\frac{\epsilon }{h} \bigg( \exp({\bm{\lambda}}_{0}^{\top} h / \epsilon ) \exp(-\bm{C}_{k}/2 \epsilon) \exp({\bm{\lambda}}_1h / \epsilon )\bigg)\bigg \}, 
\label{DualJKO}
\end{align} 
where 
\begin{eqnarray}
	F^{\star}(\bm{y}) :=\underset{\bx\in\mathbb{R}^{n}} \sup \: \{\langle\bm{y},\bx\rangle - F(\bx)  \}
\label{LFtransform}	
\end{eqnarray}
is the \emph{Legendre-Fenchel transform} of the free energy $F(\cdot)$ given by (\ref{FPKFreeEnergy1}). 
Next, we derive the first order optimality conditions for (\ref{DualJKO}), and then provide an algorithm to solve the same. 

\subsection{Conditions for Optimality}
Given the vectors $\brh_{k-1}, \bm{\psi}_{k-1}$, the matrix $\bm{C}_{k}$, and the positive scalars $\beta, h, \epsilon$ in (\ref{DualJKO}), let
\begin{alignat}{3}
&\bm{y}:=\exp({\bm{\lambda}}_0h / \epsilon ), \quad &&\bm{z}:=\exp({\bm{\lambda}}_{1}h / \epsilon ), \label{yzmaps}\\
&\bm{\Gamma}_{k} := {\exp(-\bm{C}_{k}/2\epsilon}), \quad &&\bm{\xi}_{k-1} := \exp(-\beta \bm{\psi}_{k-1} - \bm{1}). \label{Gammaxidef}
\end{alignat}
The following result provides a way of computing $\bm{\lambda}_0^{\rm{opt}},\bm{\lambda}_1^{\rm{opt}}$ in (\ref{DualJKO}), and consequently $\brh_{k}$ in (\ref{EntropyRegJKO}).
\begin{theorem}\label{FixedPtRecursionThm}
The vectors $\bm{\lambda}_0^{\rm{opt}},\bm{\lambda}_1^{\rm{opt}}$ in (\ref{DualJKO}) can be found by solving for $\bm{y}$ and $\bm{z}$ from the following system of equations:
\begin{subequations}
\begin{align}
\bm{y}\odot \left(\bm{\Gamma}_{k}\bm{z}\right) &=\bm{\rho}_{k-1}, \label{FixedPtRecursion:a}\\
\bm{z}\odot \left({\bm{\Gamma}_{k}}^{\top} \bm{y}\right) &=\bm{\xi}_{k-1}\odot \bm{z}^{-\frac{\beta\epsilon}{h}}, \label{FixedPtRecursion:b}
\end{align}
\label{FixedPtRecursion}
\end{subequations}
and then inverting the maps (\ref{yzmaps}). The vector $\brh_{k}$ in (\ref{EntropyRegJKO}), i.e., the proximal update (Fig. \ref{fig:ProxSchematic}) can then be obtained as
\begin{align}
\brh_{k} = \bm{z}^{\rm{opt}}\odot \left({\bm{\Gamma}_{k}}^{\top} \bm{y}^{\rm{opt}}\right),
\label{ProxUpdate}
\end{align}
where $\left(\bm{y}^{\rm{opt}},\bm{z}^{\rm{opt}}\right)$ denotes the solution of (\ref{FixedPtRecursion}).
\end{theorem} 
\begin{proof}
By (\ref{FPKFreeEnergy1}) and (\ref{LFtransform}), we have 
\begin{align}
F^{\star}(\bm{\lambda}) = \underset{\brh \in R^{N}} \sup  \big \{{ \bm{\lambda}^{\top} \bm{\brh} -  \bm{\psi}^{\top} \bm{\brh}-\beta^{-1} \bm{\brh}^{\top} \log \bm{\brh} } \big \}. 
\label{DiscreteLF}
\end{align}
We seek an explicit algebraic expression of (\ref{DiscreteLF}) to be substituted in (\ref{DualJKO}). Setting the gradient of the objective function in $ (\ref{DiscreteLF})$ w.r.t. $\brh$ to zero, and solving for $\brh$ yields
\begin{align}
\brh_{\max} =  \exp(\beta(\bm{\lambda}-\bm{\psi})-\bm{1}). 
\label{rhomax}
\end{align}
Substituting (\ref{rhomax}) back into (\ref{DiscreteLF}), results
\begin{align}
F^{\star}(\bm{\lambda}) = \beta^{-1}\bm{1}^{\top} \exp(\beta(\bm{\lambda}-\bm{\psi})-\bm{1}). 
\label{explicitLegendre}
\end{align}
Fixing $\bm{\lambda}_1$, and taking the gradient of the objective in (\ref{DualJKO}) w.r.t. $\bm{\lambda}_0$, gives (\ref{FixedPtRecursion:a}). Likewise, fixing $\bm{\lambda}_0$, and taking the gradient of the objective in (\ref{DualJKO}) w.r.t. $\bm{\lambda}_1$ gives 
\begin{align}
\nabla_{\bm{\lambda}_1} F^{\star}(-{\bm{\lambda}_1} )=\bm{z}\odot \left({\bm{\Gamma}_{k}}^{\top} \bm{y}\right). 
\label{gradrecur}
\end{align}
Using (\ref{explicitLegendre}) to simplify the left-hand-side of (\ref{gradrecur}) results in (\ref{FixedPtRecursion:b}). To derive (\ref{ProxUpdate}), notice that combining the last equality constraint in (\ref{CouplingMatrixConstr}) with (\ref{SinkhornFormInnerArgmin}), (\ref{yzmaps}) and (\ref{Gammaxidef}) gives \[\brh_{k} = (\bm{M}^{\rm{opt}})^{\top}\bm{1} = \sum_{j=1}^{N} m^{\rm{opt}}(j,i) = \bm{z}(i)\sum_{j=1}^{N}\bm{\Gamma}_{k}(j,i)\bm{y}(j),\] 
which is equal to $\bm{z}\odot\bm{\Gamma}_{k}^{\top}\bm{y}$, as claimed.
\end{proof}

\subsection{Algorithm}

\subsubsection{Proximal recursion}
We now propose a block co-ordinate iteration scheme to solve (\ref{FixedPtRecursion}). Specifically, the proposed procedure, which we call \textproc{ProxRecur}, and detail in Algortihm \ref{ProxRecur}, takes $\brh_{k-1}$ as input and returns the proximal update $\brh_{k}$ as output for $k=1,2,\hdots$. In addition to the data $\brh_{k-1}, \bm{\psi}_{k-1}, \bm{C}_{k}, \beta, h, \epsilon, N$, the Algorithm \ref{ProxRecur} requires two parameters as user input: numerical tolerance $\delta$, and maximum number of iterations $L$. The computation in Algorithm \ref{ProxRecur}, as presented, involves making an initial guess for the vector $\bm{z}$ and then updating $\bm{y}$ and $\bm{z}$ until convergence. 
\vspace*{-0.1in}
\begin{algorithm}
    \caption{Proximal recursion to compute $\brh_{k}$ from $\brh_{k-1}$}
    \label{ProxRecur}
    \begin{algorithmic}[1] 
        \Procedure{ProxRecur}{$\brh_{k-1}$, $\bm{\psi}_{k-1}$, $\bm{C}_{k}$, $\beta$, $h$, $\epsilon$, $N$, $\delta$, $L$}
            \State $\bm{\Gamma}_{k}\gets {\exp(-\bm{C}_{k}/2\epsilon})$
            \State $\bm{\xi} \gets \exp(-\beta \bm{\psi}_{k-1} - \bm{1})$
            
            \State $\bm{z}_{0} \gets {\rm{rand}}_{N\times 1}$\Comment{initialize}
            \State $\bm{z} \gets \left[\bm{z}_{0}, \bm{0}_{N\times (L-1)}\right]$ 
            \State $\bm{y} \gets \left[\brh_{k-1} \oslash \left(\bm{\Gamma}_{k}\bm{z}_{0}\right), \bm{0}_{N\times (L-1)}\right]$
            \State $\ell=1$ \Comment{iteration index}
            
            \While{$\ell \leq L$}
                \State $\bm{z}(:,\ell+1) \gets \left(\bm{\xi}_{k-1} \oslash \left(\bm{\Gamma}_{k}^{\top}\bm{y}(:,\ell)\right)\right)^{\frac{1}{1+\beta\epsilon/h}}$
                \State $\bm{y}(:,\ell+1) \gets \bm{\brh}_{k-1} \oslash \left( \bm{\Gamma}_{k} \bm{z}(:,\ell+1)\right) $
                \If{$\parallel \bm{y}(:,\ell+1) - \bm{y}(:,\ell)\parallel < \delta \;\And\; \parallel \bm{z}(:,\ell+1)- \bm{z}(:,\ell)\parallel < \delta$ } \Comment{error within tolerance}
                \State break
                \Else
                \State $\ell \gets \ell + 1$
                \EndIf
            \EndWhile\label{euclidendwhile}
            \State \textbf{return} $\brh_{k} \gets \bm{z}(:,\ell) \odot \left( \bm{\Gamma}_{k}^{\top} \bm{y}(:,\ell) \right)$
        \EndProcedure
    \end{algorithmic}
\end{algorithm}
\vspace*{-0.1in}

Several questions arise: how can one ensure that such a procedure converges? Also, even if convergence can be guaranteed, is the rate fast in practice? The latter issue is important since the time-step $h$ in the JKO scheme is small, and during the computation of Algorithm 1, the physical time is ``frozen". We will establish the convergence guaranteed by showing certain contractive properties of the recursion given in Algorithm \ref{ProxRecur}. Before doing so, we next outline the overall algorithmic setup to implement the proximal recursion over probability weighted scattered point cloud data.

\begin{figure}[h]
\centering
\includegraphics[width=.65\linewidth]{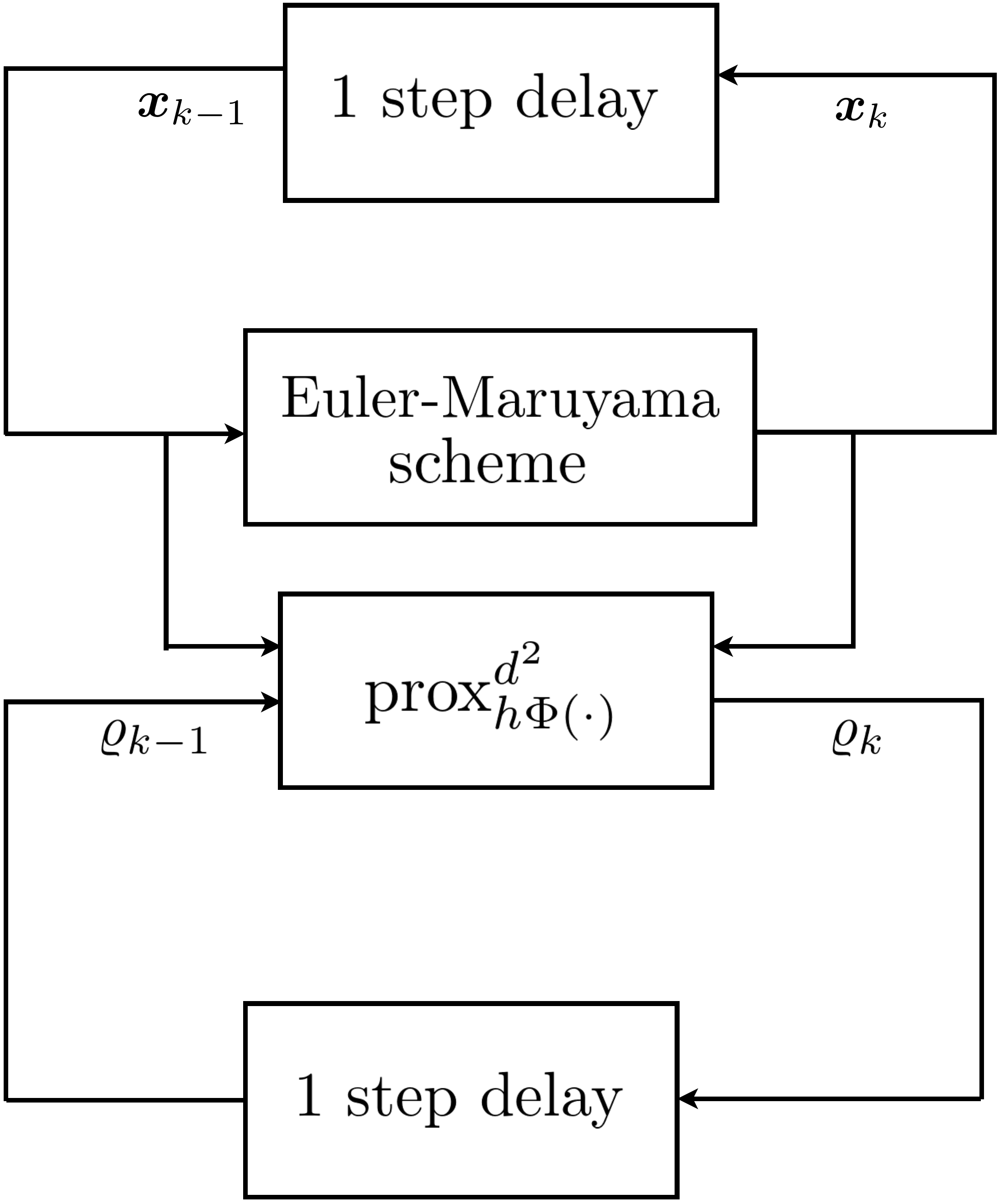}
\caption{\small{Schematic of the proposed algorithmic setup for propagating the joint state PDF as probability weighted scattered point cloud $\{\bx_{k}^{i},\varrho_{k}^{i}\}_{i=1}^{N}$. The location of the points $\{\bx_{k}^{i}\}_{i=1}^{N}$ are updated via Euler-Maruyama scheme; the corresponding probability weights are updated via Algorithm \ref{ProxRecur}. 
}}
\vspace*{-0.25in}
\label{BlockDiagm}
\end{figure}

\subsubsection{Overall scheme}
Samples from the \emph{known} initial joint PDF $\rho_{0}$ are generated as point cloud $\{\bx_{0}^{i},\varrho_{0}^{i}\}_{i=1}^{N}$. Then for $k=1,2,\hdots$, the point clouds $\{\bx_{k}^{i},\varrho_{k}^{i}\}_{i=1}^{N}$ are updated as shown in Fig. \ref{BlockDiagm}. Specifically, the state vectors are updated via Euler-Maruyama scheme applied to the underlying SDE; the corresponding probability weights are updated via Algorithm \ref{ProxRecur}. Notice that computing $\bm{C}_{k}$ requires both $\{\bm{x}_{k-1}^{i}\}_{i=1}^{N}$ and $\{\bm{x}_{k}^{i}\}_{i=1}^{N}$, and that $\bm{C}_{k}$ needs to be passed as input to Algorithm \ref{ProxRecur}. Thus, the execution of Euler-Maruyama scheme precedes that of Algorithm \ref{ProxRecur}.

\subsection{Convergence}
The following Definition \ref{DefnThompson} and Proposition \ref{PropOrderPreservingMap} will be useful in proving Theorem \ref{ThmConvergence} that follows which establishes the convergence of Algorithm \ref{ProxRecur}.
\begin{definition}(\textbf{Thompson metric})
\label{DefnThompson}
Consider $\bm{z}, \widetilde{\bm{z}} \in \mathcal{K}$, where $\mathcal{K}$ is a non-empty open convex cone. Further, suppose that $\mathcal{K}$ is a normal cone, i.e., there exists constant $\alpha$ such that $\parallel \bm{z} \parallel \leq \alpha \parallel \widetilde{\bm{z}} \parallel$ for $\bm{z} \leq \widetilde{\bm{z}}$. Thompson \cite{thompson1963certain} proved that	$\mathcal{K}$ is a complete metric space w.r.t. the so-called \emph{Thompson metric} given by
\[d_{\rm{T}}\left(\bm{z},\widetilde{\bm{z}}\right) := \max\{\log \gamma(\bm{z}/\widetilde{\bm{z}}), \log \gamma(\widetilde{\bm{z}}/\bm{z})\},\]
where $\gamma(\bm{z}/\widetilde{\bm{z}}) := \inf\{c>0 \mid \bm{z} \leq c\widetilde{\bm{z}}\}$. In particular, if $\mathcal{K}\equiv \mathbb{R}^{n}_{>0}$ (positive orthant of $\mathbb{R}^{n}$), then
\begin{eqnarray}
d_{\rm{T}}\left(\bm{z},\widetilde{\bm{z}}\right) = \log\max\bigg\{\max_{i=1,\hdots,n}\left(\frac{\bm{z}_{i}}{\widetilde{\bm{z}}_{i}}\right), \max_{i=1,\hdots,n}\left(\frac{\widetilde{\bm{z}}_{i}}{\bm{z}_{i}}\right)\bigg\}.
\label{ThompsonPosOrthant}	
\end{eqnarray}
\end{definition}
\begin{proposition}\label{PropOrderPreservingMap}\cite[Proposition 3.2]{lim2009nonlinear},\cite{nussbaum1988hilbert}
Let $\mathcal{K}$ be an open, normal, convex cone and let $\bm{\phi} : \mathcal{K} \mapsto \mathcal{K}$ be an order preserving homogeneous map of degree $r \geq 0$. Then, for all $\bm{z}, \widetilde{\bm{z}} \in \mathcal{K}$, we have
\[d_{\rm{T}}\left(\bm{\phi}(\bm{z}),\bm{\phi}(\widetilde{\bm{z}})\right) \leq r d_{\rm{T}}\left(\bm{z},\widetilde{\bm{z}}\right).\]
In particular, if $r\in[0,1)$, then the map $\bm{\phi}(\cdot)$ is strictly contractive in the Thompson metric $d_{\rm{T}}$, and admits unique fixed point in $\mathcal{K}$.
\end{proposition}

Using (\ref{ThompsonPosOrthant}) and Proposition \ref{PropOrderPreservingMap}, we establish the convergence result below.
\begin{theorem}\label{ThmConvergence}
Consider the notations in (\ref{yzmaps})-(\ref{Gammaxidef}), and those in Algorithm \ref{ProxRecur}. The iteration
\begin{align}
\bm{z}(:,\ell+1) &= \left(\bm{\xi}_{k-1} \oslash \left(\bm{\Gamma}_{k}^{\top}\bm{y}(:,\ell)\right)\right)^{\frac{1}{1+\beta\epsilon/h}}\nonumber\\
&= \left(\bm{\xi}_{k-1} \oslash \left(\bm{\Gamma}_{k}^{\top}\bm{\brh}_{k-1} \oslash \left( \bm{\Gamma}_{k} \bm{z}(:,\ell)\right)\right)\right)^{\frac{1}{1+\beta\epsilon/h}}
\label{ConvThmIter}	
\end{align}
for $\ell = 1, 2, \hdots$, is strictly contractive in the Thompson metric (\ref{ThompsonPosOrthant}) on $\mathbb{R}^{n}_{>0}$, and admits unique fixed point $\bm{z}^{\rm{opt}} \in \mathbb{R}^{n}_{>0}$. 
\end{theorem}
\begin{proof}
Rewriting (\ref{ConvThmIter}) as
\begin{eqnarray*}
\bm{z}(:,\ell+1) = \left(\bm{\xi}_{k-1} \oslash \left(\bm{\Gamma}_{k}^{\top}\bm{\brh}_{k-1}\right) \oslash \left( \bm{\Gamma}_{k} \bm{z}(:,\ell)\right)\right)^{\frac{1}{1+\beta\epsilon/h}},	
\end{eqnarray*}
and letting $\bm{\eta} \equiv \bm{\eta}_{k,k+1} := \bm{\xi}_{k-1} \oslash \left(\bm{\Gamma}_{k}^{\top}\bm{\brh}_{k-1}\right)$, we notice that iteration (\ref{ConvThmIter}) can be expressed as a cone preserving composite map $\bm{\theta} := \bm{\theta}_{1} \circ \bm{\theta}_{2} \circ \bm{\theta}_{3} \circ \bm{\theta}_{4}$, where $\bm{\theta} : \mathbb{R}^{n}_{>0} \mapsto \mathbb{R}^{n}_{>0}$, given by 
\begin{eqnarray}
\bm{z}(:,\ell+1) = \bm{\theta}\left(\bm{z}(:,\ell)\right) = \bm{\theta}_{1} \circ \bm{\theta}_{2} \circ \bm{\theta}_{3} \circ \bm{\theta}_{4} \: \left(\bm{z}(:,\ell)\right),
\end{eqnarray}
and $\bm{\theta}_{1}(\bm{z}) := \bm{z}^{\frac{1}{1+\beta\epsilon/h}}$, $\bm{\theta}_{2}(\bm{z}) := \bm{\eta}\odot\bm{z}$, $\bm{\theta}_{3} := \bm{1} \oslash \bm{z}$, $\bm{\theta}_{4}(\bm{z}) := \bm{\Gamma}_{k} \bm{z}$. Our strategy is to prove that the composite map $\bm{\theta}$ is contractive on $\mathbb{R}^{n}_{>0}$ w.r.t. the metric $d_{\rm{T}}$.

From (\ref{Gammaxidef}), notice that since $\bm{C}_{k}(i,j)\in[0,\infty)$ we have $\bm{\Gamma}_{k}(i,j) \in (0,1]$; therefore, $\bm{\Gamma}_{k}$ is a positive linear map for each $k=1,2,\hdots$. Thus, by (linear) Perron-Frobenius theorem, the map $\bm{\theta}_{4}$ is contractive on $\mathbb{R}^{n}_{>0}$ w.r.t. $d_{\rm{T}}$. The map $\bm{\theta}_{3}$ involves element-wise inversion, which is an isometry on $\mathbb{R}^{n}_{>0}$ w.r.t. $d_{\rm{T}}$. Also, the map $\bm{\theta}_{2}$ is an isometry by Definition \ref{DefnThompson}. As for the map $\bm{\theta}_{1}$, notice that the quantity $r := 1/(1 + \beta\epsilon/h) \in (0,1)$ since $\beta\epsilon/h > 0$. Therefore, the map $\bm{\theta}_{1}(\bm{z}) := \bm{z}^{r}$ (element-wise exponentiation) is monotone (order preserving) and homogeneous of degree $r\in(0,1)$ on $\mathbb{R}^{n}_{>0}$. By Proposition \ref{PropOrderPreservingMap}, the map $\bm{\theta}_{1}(\bm{z})$ is strictly contractive. Thus, the composition
\[\bm{\theta} = \underbrace{\bm{\theta}_{1}}_{\text{strictly contractive}} \circ \underbrace{\bm{\theta}_{2}}_{\text{isometry}} \circ \underbrace{\bm{\theta}_{3}}_{\text{isometry}} \circ \underbrace{\bm{\theta}_{4}}_{\text{contractive}}\]
is strictly contractive w.r.t. $d_{\rm{T}}$, and (by Banach contraction mapping theorem) admits unique fixed point $\bm{z}^{\rm{opt}}$ in $\mathbb{R}^{n}_{>0}$. 
\end{proof}
\begin{corollary}\label{CorrConv}
The Algorithm \ref{ProxRecur} converges to unique fixed point $(\bm{y}^{\rm{opt}},\bm{z}^{\rm{opt}}) \in \mathbb{R}^{n}_{>0} \times \mathbb{R}^{n}_{>0}$.	
\end{corollary}
\begin{proof}
Since $\bm{y}(:,\ell+1) = \bm{\brh}_{k-1} \oslash \left( \bm{\Gamma}_{k} \bm{z}(:,\ell+1)\right)$, the $\bm{z}$ iterates converge to unique fixed point $\bm{z}^{\rm{opt}}\in\mathbb{R}^{n}_{>0}$ (by Theorem \ref{ThmConvergence}), and the linear maps $\bm{\Gamma}_{k}$ are contractive (by Perron-Frebenius theory, as before), consequently the $\bm{y}$ iterates also converge to unique fixed point $\bm{y}^{\rm{opt}}\in\mathbb{R}^{n}_{>0}$. Hence the statement.
\end{proof}


%
%
%
%
%

\begin{figure*}[t]
\centering
\includegraphics[width=.98\linewidth]{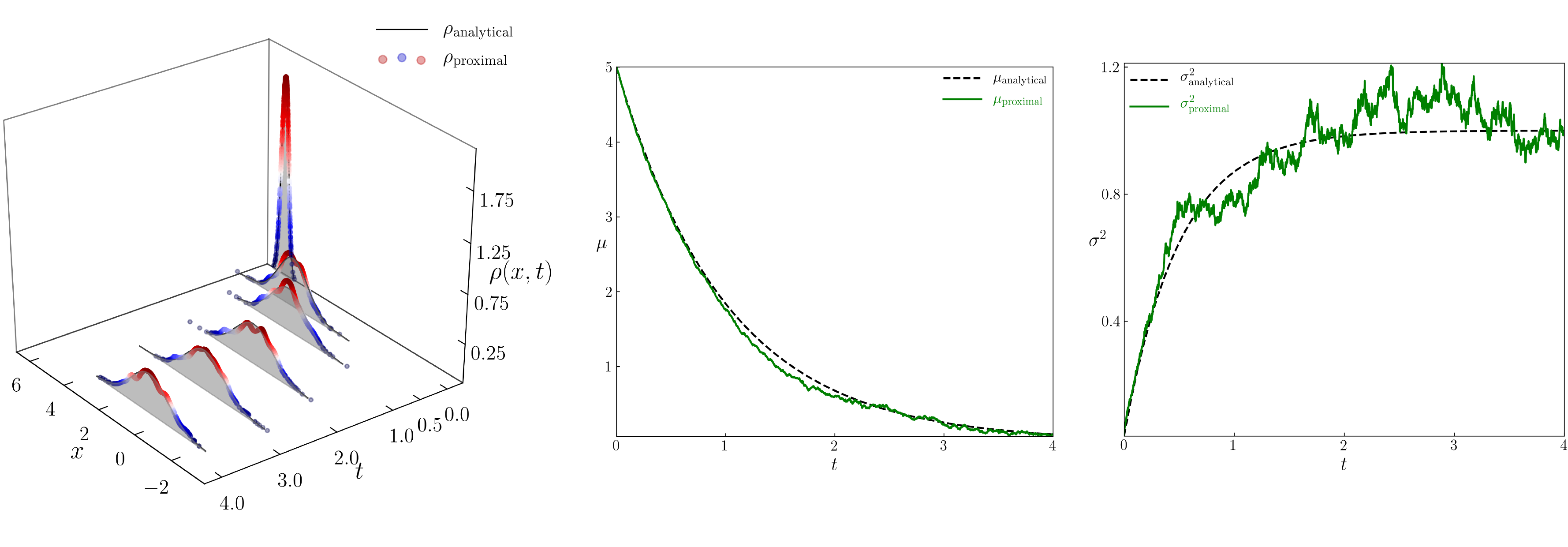}
\vspace*{-0.15in}
\caption{\small{Comparison of the analytical and proximal solutions of the FPK PDE for (\ref{OU}) with time step $h=10^{-3}$, and with parameters $a=1$, $\beta=1$, $\epsilon = 5\times 10^{-2}$. Shown above are the time evolution of the (\emph{left}) PDFs, (\emph{middle}) means, and (\emph{right}) variances.}}
\vspace*{-0.2in}
\label{1dOU}
\end{figure*}

\section{Numerical Simulation}
In this section, we apply the algorithmic setup proposed in Section 	IV.B to few examples illustrating the numerical approach. Our examples involve systems which are already in JKO canonical form (Section III), as well as those which can be transformed to such form by non-obvious change of coordinates. 

\subsection{Linear Gaussian System}
For an It\^{o} SDE of the form 
\begin{eqnarray}
\differential \bx = \bm{A}\bx \: \differential t + \bm{B} \: \differential \bw, \label{linearSDE}
\end{eqnarray}
it is well known that if $\bx_{0} := \bx(t=0) \sim \mathcal{N}(\bm{\mu}_0,\bm{\Sigma}_0)$,  then the transient joint PDFs $\rho(\bx,t) = \mathcal{N}(\bm{\mu}(t),\bm{\Sigma}(t))$ where the vector-matrix pair $\left(\bm{\mu}(t),\bm{\Sigma}(t)\right)$ evolve according to the ODEs 
\begin{subequations}
\begin{align}
\dot{\bm{\mu}}(t) &= \bm{A}\bm{\mu}, \quad \bm{\mu}(0)=\bm{\mu}_0, \label{meanODE}\\
\dot{\bm{\Sigma}}(t) &= \bm{A}\bm{\Sigma(t)} + \bm{A} \bm{\Sigma(t)}^{\top} + \bm{B} \bm{B}^{\top}, \quad \bm{\Sigma}(0) = \bm{\Sigma}_0.\label{covODE}	
\end{align}
\label{MeanCovODE}	
\end{subequations}
We benchmark the numerical results produced by the proposed proximal algorithm vis-\`{a}-vis the above analytical solutions. We consider the following two sub-cases of (\ref{linearSDE}).

\subsubsection{Ornstein-Uhlenbeck Process}
We consider the 1D system 
\begin{eqnarray}
\differential x = -ax \: \differential t + \sqrt{2\beta^{-1}} \differential w, \quad a,\beta > 0,	
\label{OU}
\end{eqnarray}
which is in JKO canonical form with $\psi(x) = \frac{1}{2}ax^{2}$. We generate $N=400$ samples from the initial PDF $\rho_{0} = \mathcal{N}(\mu_{0},\sigma_{0}^{2})$ with $\mu_{0}=5$ and $\sigma_{0}^{2}=4\times 10^{-2}$, and apply the proposed proximal recursion for (\ref{OU}) with time step $h=10^{-3}$, and with parameters $a=1$, $\beta=1$, $\epsilon = 5\times 10^{-2}$. For implementing Algorithm \ref{ProxRecur}, we set tolerance $\delta = 10^{-3}$, and maximum number of iterations $L=100$. Fig. \ref{1dOU} shows that the PDF point clouds generated by the proximal recursion match with the analytical PDFs $\mathcal{N}\left(\mu_{0}\exp(-at),(\sigma_{0}^{2}-\frac{1}{a\beta})\exp(-2at) + \frac{1}{a\beta}\right)$, and the mean-variance trajectories (computed from the numerical integration of the point cloud data) match with the corresponding analytical solutions.

\subsubsection{Multivariate LTI}
We next consider the multivariate case (\ref{linearSDE}) where the pair $(\bm{A},\bm{B})$ is assumed to be controllable, and the matrix $\bm{A}$ is Hurwitz (not necessarily symmetric). Under these assumptions, the stationary PDF is $\mathcal{N}(\bm{0},\bm{\Sigma}_{\infty})$ where $\bm{\Sigma}_{\infty}$ is the unique stationary solution of (\ref{covODE}) that is guaranteed to be symmetric positive definite. However, it is not apparent whether (\ref{linearSDE})  can be expressed in the form (\ref{ItoGradient}), since for non-symmetric $\bm{A}$, there does \emph{not} exist constant symmetric positive definite matrix $\bm{\Psi}$ such that $\bm{Ax} = -\nabla\bx^{\top} \bm{\Psi} \bx$, i.e., the drift vector field does not admit a natural potential. Thus, implementing the JKO scheme for (\ref{linearSDE}) is non-trivial in general.

In a recent work \cite{halder2017gradient}, two successive \emph{time-varying} co-ordinate transformations were given which can bring (\ref{linearSDE}) in the form (\ref{ItoGradient}), thus making it amenable to the JKO scheme. We apply these change-of-coordinates to (\ref{linearSDE}) with 
\[\bm{A} = \begin{pmatrix}
 -10 & 5\\
 -30 & 0	
 \end{pmatrix}, \quad \bm{B} = \begin{pmatrix}
 2\\2.5	
 \end{pmatrix},
\]
which satisfy the stated assumptions on $(\bm{A},\bm{B})$, and implement the proposed proximal recursion on this transformed co-ordinates with $N=400$ samples generated from the initial PDF $\rho_{0} = \mathcal{N}(\bm{\mu}_{0},\bm{\Sigma}_{0})$, where $\bm{\mu}_{0}=(4,4)^{\top}$ and $\bm{\Sigma}_{0}=4\bm{I}_{2}$. As before, we set $\delta = 10^{-3}, L=100, h=10^{-3}, \beta=1, \epsilon = 5\times 10^{-2}$. Once the proximal updates are done, we transform back the probability weighted scattered point cloud to the original state space co-ordinates via change-of-measure formula associated with the known co-ordinate transforms \cite[Section III.B]{halder2017gradient}. Fig. \ref{Hurwitz} shows the resulting point clouds superimposed with the contour plots for the analytical solutions $\mathcal{N}(\bm{\mu}(t),\bm{\Sigma}(t))$ given by (\ref{MeanCovODE}). Figs. \ref{Hurwitzmean} and \ref{Hurwitzcovar} compare the respective mean and covariance evolution. We point out that the change of co-ordinates in \cite{halder2017gradient} requires implementing the JKO scheme in a time-varying rotating frame (defined via exponential of certain time varying skew-symmetric matrix) that depends on the stationary covariance $\bm{\Sigma}_{\infty}$. As a consequence, the stationary covariance resulting from the proximal recursion oscillates about the true stationary value.

\begin{figure}[h]
\centering
\includegraphics[width=\linewidth]{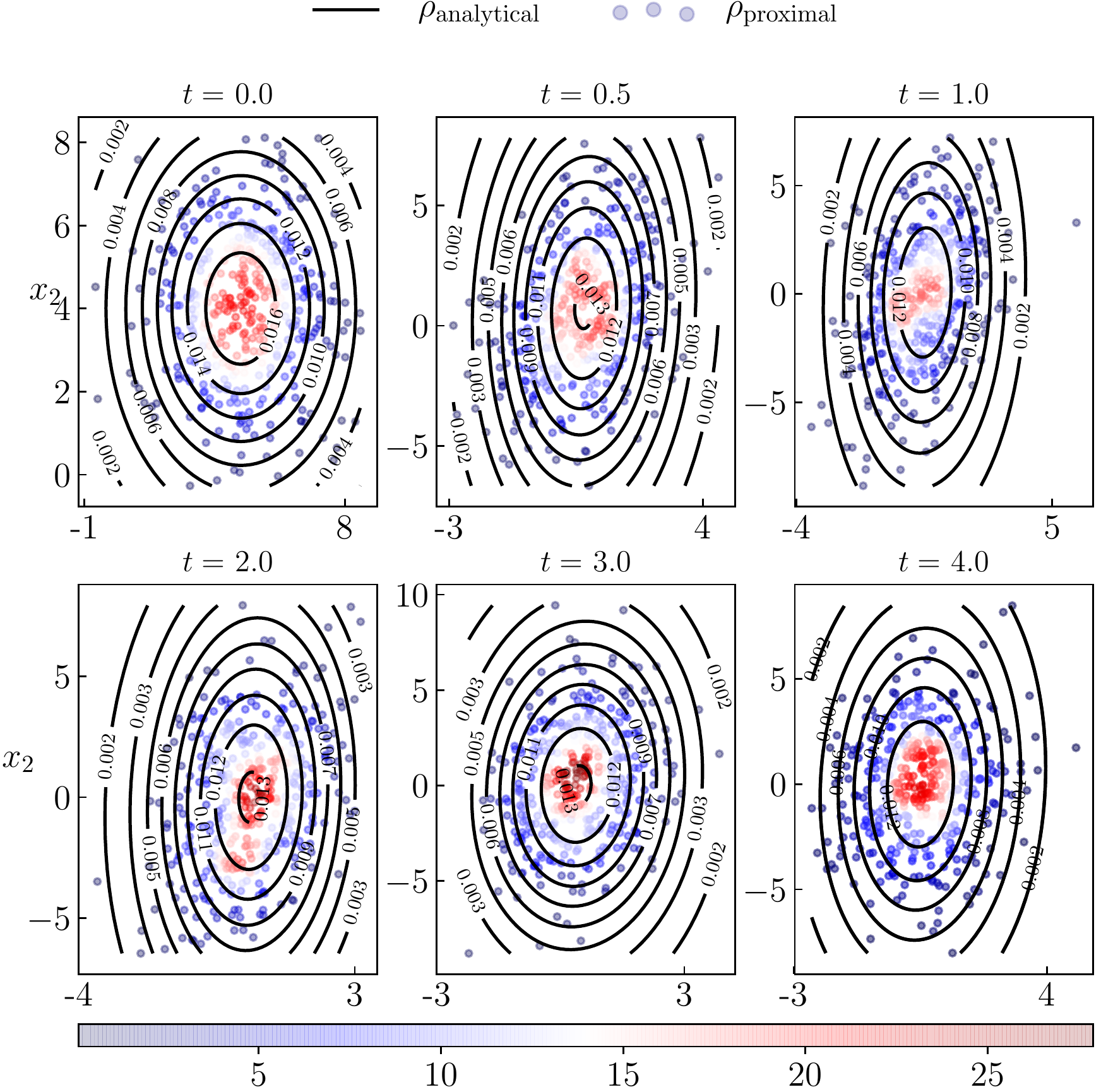}
\vspace*{-0.2in}
\caption{\small{Comparison of the analytical (\emph{contour plots}) and proximal (\emph{weighted scattered point cloud}) joint PDFs of the FPK PDE for (\ref{linearSDE}) with time step $h=10^{-3}$, and with parameters $\beta=1, \epsilon = 5\times 10^{-2}$. Simulation details are given in Section V.A.2. The color (\emph{red = high, blue = low}) denotes the joint PDF value obtained via proximal recursion at a point at that time (see colorbar).}}
\vspace*{-0.1in}
\label{Hurwitz}
\end{figure}

\begin{figure}[h]
\centering
\includegraphics[width=.85\linewidth]{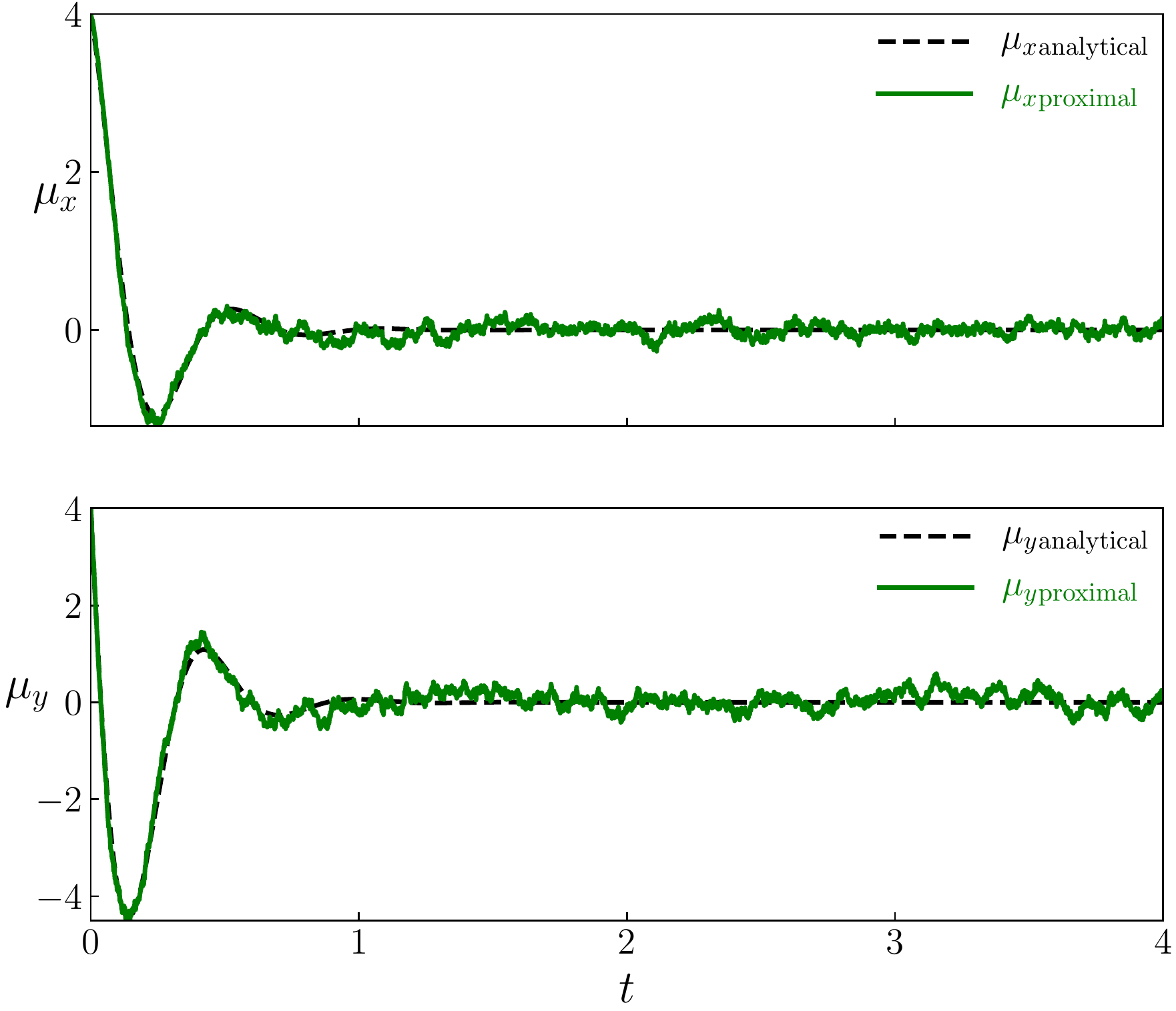}
\vspace*{-0.1in}
\caption{\small{Comparison of the components of the mean vectors from analytical (\emph{dashed}) and proximal (\emph{solid}) computation of the joint PDFs for (\ref{linearSDE}) with time step $h=10^{-3}$, and with parameters $\beta=1, \epsilon = 5\times 10^{-2}$. Simulation details are given in Section V.A.2.}}
\vspace*{-0.2in}
\label{Hurwitzmean}
\end{figure}

\begin{figure}[h] 
\centering
\includegraphics[width=.85\linewidth]{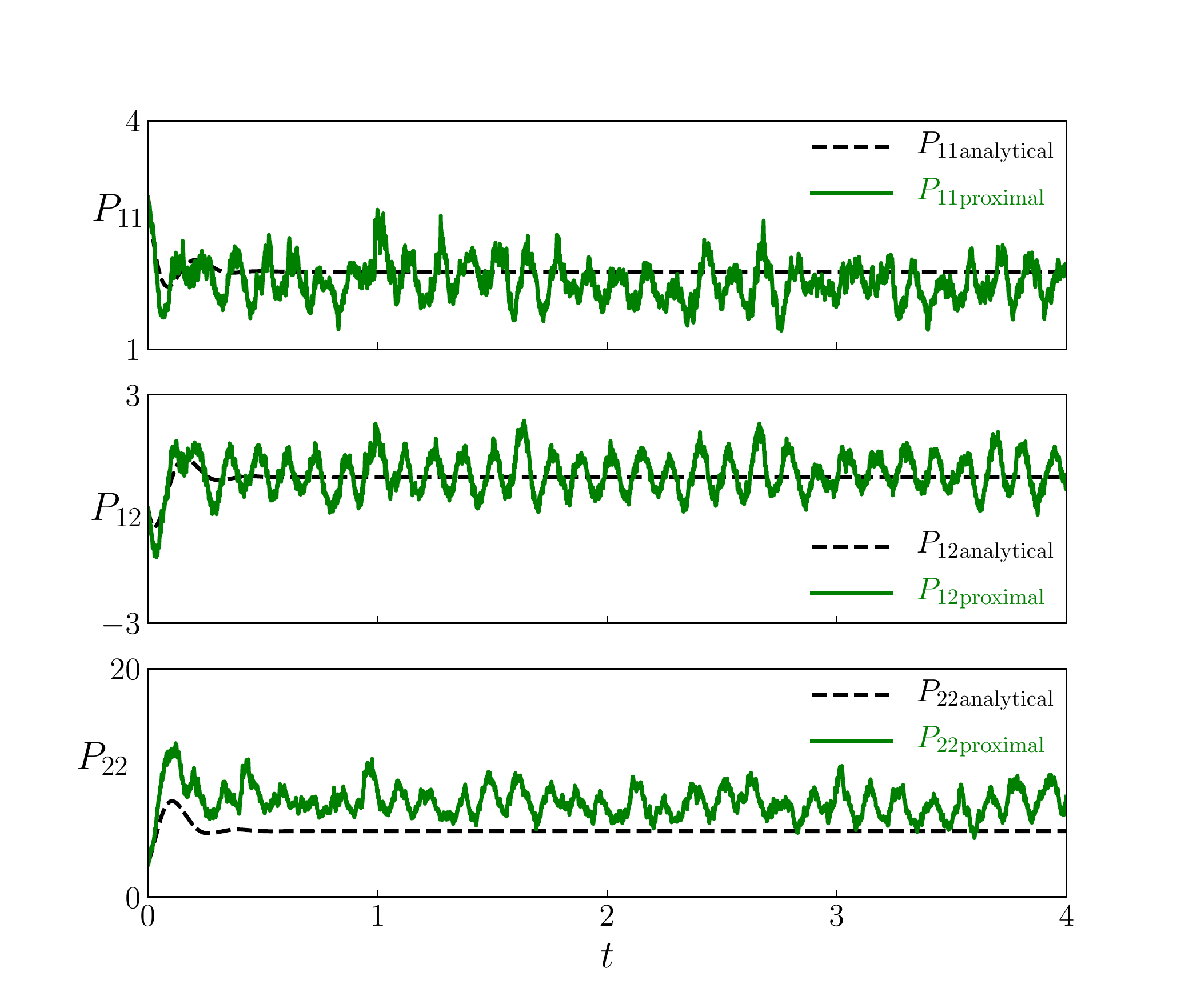}
\vspace*{-0.1in}
\caption{\small{Comparison of the components of the covariance matrices from analytical (\emph{dashed}) and proximal (\emph{solid}) computation of the joint PDFs for (\ref{linearSDE}) with time step $h=10^{-3}$, and with parameters $\beta=1, \epsilon = 5\times 10^{-2}$. Simulation details are given in Section V.A.2.}}
\vspace*{-0.1in}
\label{Hurwitzcovar}
\end{figure}

\begin{figure}[h]
\centering
\includegraphics[width=0.82\linewidth]{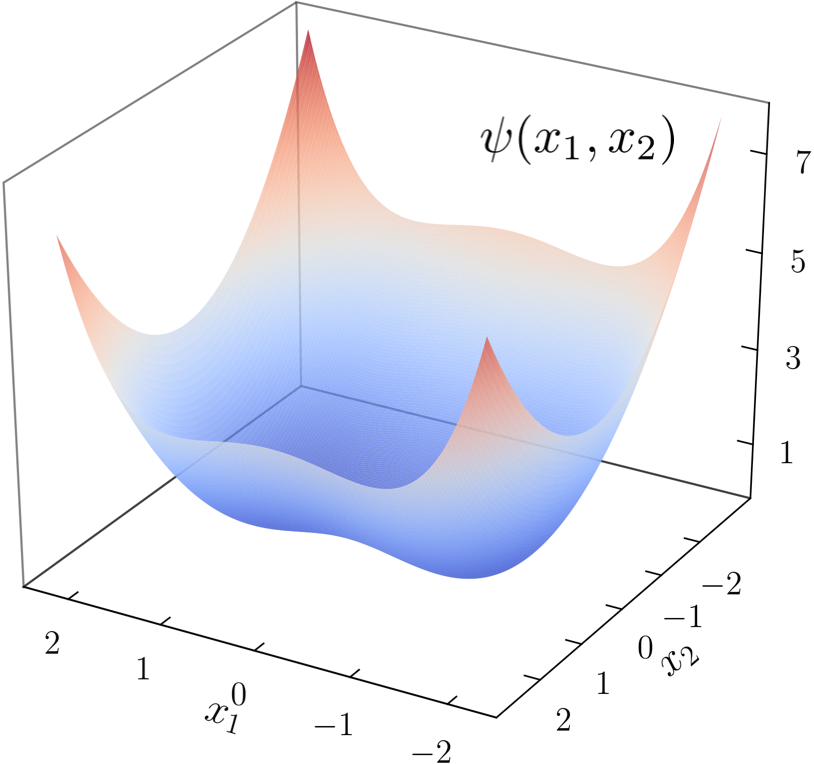}
\vspace*{-0.1in}
\caption{\small{The drift potential $\psi(x_{1},x_{2}) = \dfrac{1}{4}(1+x_1^4) + \dfrac{1}{2}(x_2^2-x_1^2)$ used in the example given in Section V.B.}}
\vspace*{-0.25in}
\label{2dpotential}
\end{figure}

\begin{figure}[h]
\centering
\includegraphics[width=0.95\linewidth]{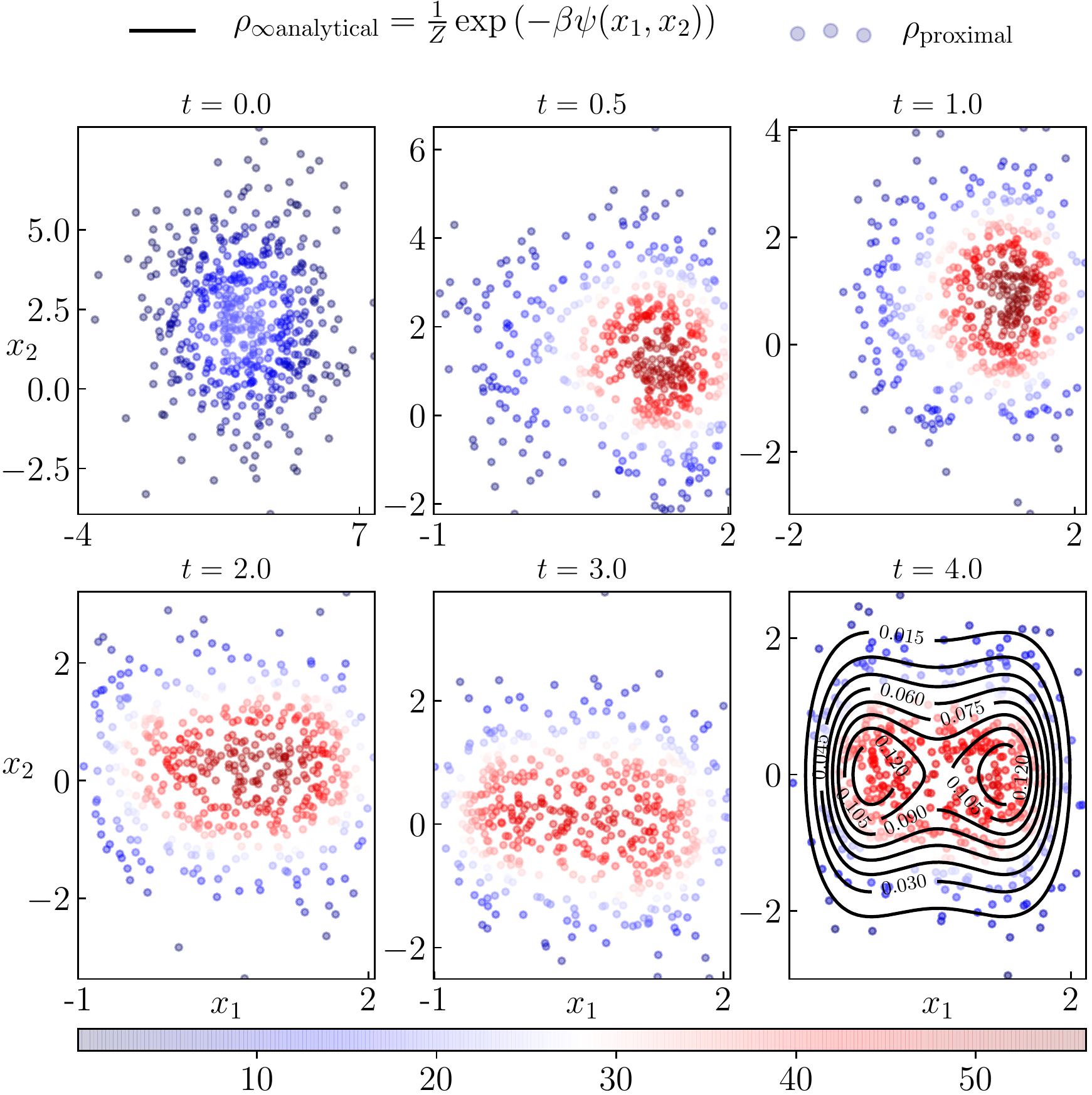}
\caption{\small{The proximal (\emph{weighted scattered point cloud}) joint PDFs of the FPK PDE (\ref{FPKgradient}) with the drift potential shown in Fig. \ref{2dpotential}, time step $h=10^{-3}$, and with parameters $\beta=1, \epsilon = 5\times 10^{-2}$. Simulation details are given in Section V.B. The color (\emph{red = high, blue = low}) denotes the joint PDF value obtained via proximal recursion at a point at that time (see colorbar).}}
\vspace*{-0.1in}
\label{2dgrad}
\end{figure}


\subsection{Nonlinear non-Gaussian System}
Next we consider the 2D nonlinear system of the form (\ref{ItoGradient}) with $\psi(x_{1},x_{2}) = \dfrac{1}{4}(1+x_1^4) + \dfrac{1}{2}(x_2^2-x_1^2)$ (see Fig. \ref{2dpotential}). As mentioned in Section III, the stationary PDF is $\rho_{\infty}(\bx) = \kappa \exp\left(-\beta\psi(\bx)\right)$, which for our choice of $\psi$, is bimodal. The transient PDFs have no known analytical solution but can be computed using the proposed proximal recursion. For doing so, we generate $N=400$ samples from the initial PDF $\rho_{0} = \mathcal{N}(\bm{\mu}_{0},\bm{\Sigma}_{0})$ with $\bm{\mu}_{0}=(2,2)^{\top}$ and $\bm{\Sigma}_{0}=4\bm{I}_{2}$, and set $\delta = 10^{-3}, L=100, h=10^{-3}, \beta=1, \epsilon = 5\times 10^{-2}$, as before. The resulting weighted point clouds are shown in Fig. \ref{2dgrad}; it can be seen that as time progresses, the joint PDFs computed via the proximal recursion, tend to the known stationary solution $\rho_{\infty}$ (contour plots in the right bottom sub-figure in Fig. \ref{2dgrad}).

\begin{figure}[h]
\centering
\includegraphics[width=0.9\linewidth]{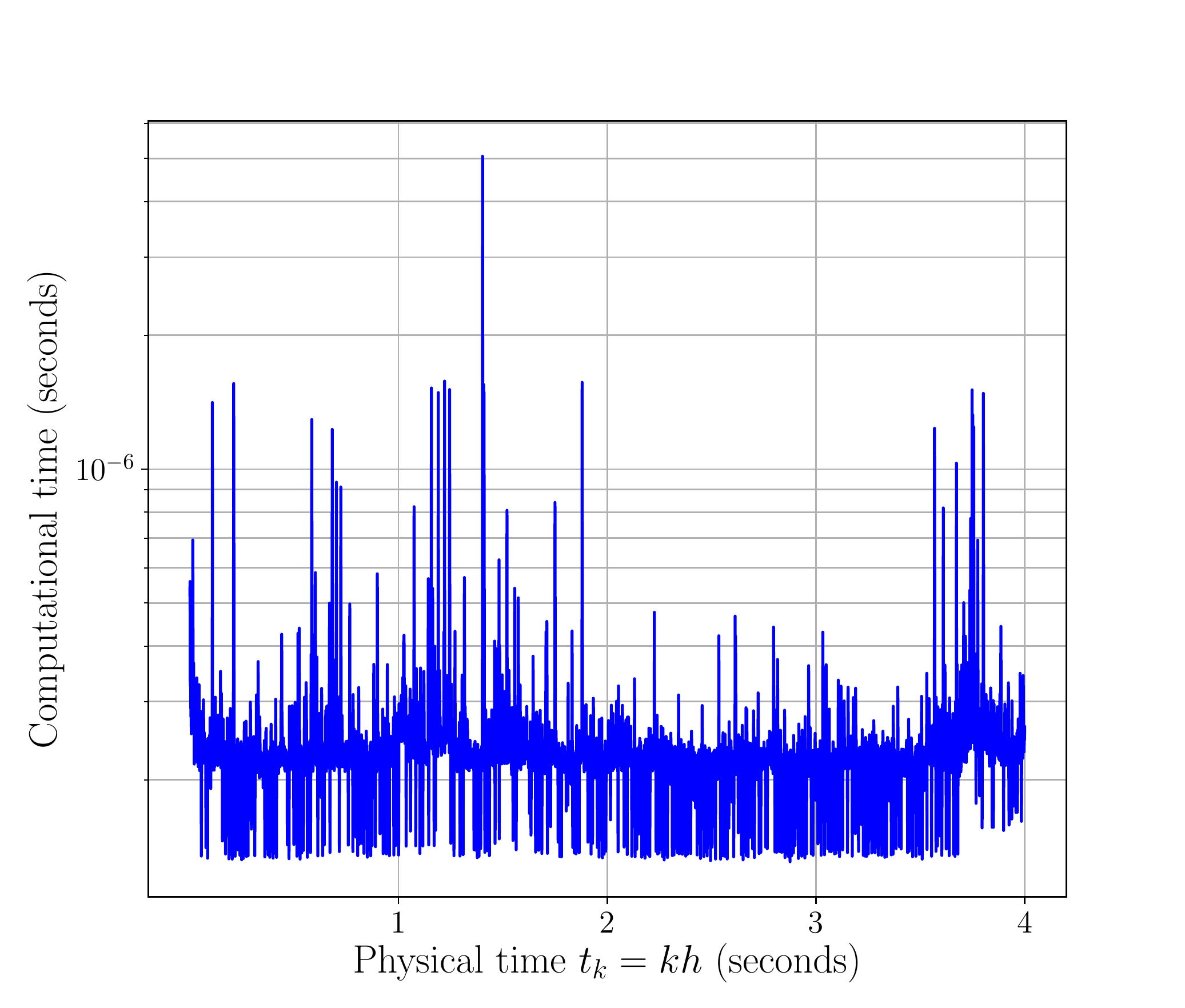}
\caption{\small{The computational times for proximal updates. Simulation details are given in Section V.B. Here, the physical time-step $h= 10^{-3}$ s, and $k=1,2,\hdots$.}}
\vspace*{-0.25in}
\label{RateOfConvProx}
\end{figure}

Fig. \ref{RateOfConvProx} shows the computational times for the proposed proximal recursions applied to the above nonlinear non-Gaussian system. Since the proposed algorithm involves sub-iterations (see while loop in Algorithm \ref{ProxRecur}) while keeping the physical time ``frozen", the convergence reported in Section IV.C must be achieved at ``sub-physical time step" level, i.e., must incur smaller than $h$ (here, $h= 10^{-3}$ s) computational time. Indeed, Fig. \ref{RateOfConvProx} shows that each proximal update takes approx. $10^{-6}$ s, or $10^{-3} h$ computational time, which demonstrates the efficacy of the proposed framework. 

\section{Conclusions}
We proposed a variational recursion to numerically solve the transient Fokker-Planck or Kolmogorov's forward equation by exploiting the underlying infinite-dimensional gradient flow structure in the manifold of PDFs. From a computational standpoint, this work develops a novel point cloud solver for performing the Otto calculus avoiding spatial discretization or function approximation. From systems-theoretic standpoint, this work contributes to an emerging research program \cite{halder2017gradient,halder2018gradient} in uncovering new geometric meanings of the equations of uncertainty propagation and filtering, and using the same to efficiently solve these equations via proximal algorithms \cite{parikh2014proximal}.





%
\bibliographystyle{IEEEtran}
\bibliography{references.bib}


\end{document}